\documentclass[a4paper,11pt,reqno]{article}

\usepackage{relsize}
\usepackage{cite}
\usepackage{color}
\usepackage{hyperref}
\hypersetup{
  colorlinks   = true, 
  urlcolor     = green, 
  linkcolor    = blue, 
  citecolor   = red 
}
\usepackage{amsmath,amsthm,amssymb}

\usepackage[margin=2.6cm]{geometry}

\makeatletter
\usepackage{comment}
\let\wfs@comment@comment\comment
\let\comment\@undefined

\usepackage{changes}
\let\wfs@changes@comment\comment
\let\comment\@undefined

\newcommand\comment{%
    \ifthenelse{\equal{\@currenvir}{comment}}
    {\wfs@comment@comment}
    {\wfs@changes@comment}%
}

\usepackage{stmaryrd}

\theoremstyle{definition}
\newtheorem{theorem}{Theorem}[section]
\newtheorem{definition}[theorem]{{{Definition}}}
\newtheorem{example}[theorem]{{{Example}}}

\newtheorem{remark}[theorem]{{{Remark}}}

\newtheorem{corollary}[theorem]{{{Corollary}}}
\newtheorem{proposition}[theorem]{{{Proposition}}}
\newtheorem{lemma}[theorem]{{{Lemma}}}

\newcommand{\numberset}{\mathbb}

\newcommand{\F}{\numberset{F}}

\newcommand{\CC}{\mathcal{C}}

\newcommand{\wt}{\textnormal{wt}}

\newcommand{\rank}{\textnormal{rank}}

\DeclareMathOperator{\supp}{supp}

\title{\textbf{On ideals in group algebras: \\an uncertainty principle and the Schur product}}

\usepackage{authblk}

\author[1]{Martino Borello}
\affil[1]{Universit\'e Paris 8, Laboratoire de G\'eom\'etrie, Analyse et Applications, LAGA,
Universit\'e Sorbonne Paris Nord, CNRS, UMR 7539, France}

\author[2]{Wolfgang Willems}
\affil[2]{Otto-von-Guericke Universit\"{a}t, Magdeburg, Germany and
Universidad del Norte, Barranquilla, Colombia}

\author[3]{Giovanni Zini}
\affil[3]{Universit\`a degli Studi di Modena e Reggio Emilia, Modena, Italy}

\date{}

\usepackage{setspace}

\usepackage{enumitem}
\setitemize{itemsep=0em}
\setenumerate{itemsep=0em}

\begin{document}

\maketitle

\begin{abstract}
In this paper we investigate some properties of ideals in group algebras of finite groups over fields. First, we highlight an important link between their dimension, their minimal Hamming distance and the group order. This is a generalized version of an uncertainty principle shown in 1992 by Meshulam. Secondly, we introduce the notion of the Schur product of ideals in group algebras and investigate the module structure and the dimension of the Schur square.
We give a structural result on ideals that coincide with their Schur square, and we provide conditions for an ideal to be such that its Schur square has the projective cover of the trivial module as a direct summand. This has particularly interesting consequences for group algebras of $p$-groups over fields of characteristic $p$.
\end{abstract}



\section*{Introduction}

Studying the algebraic structure of group algebras $KG$ in positive characteristic $p$ (which reflects many $p$-local properties of the underlying group $G$) means to a large extent studying its ideals. In this pure representation theoretical context the Hamming metric, which is naturally given on $KG$ and has a coding theoretical meaning, is not often considered. 
In this paper we consider ideals in group algebras $KG$ of finite groups $G$ over fields $K$, endowed with the Hamming metric of $KG$. Such ideals are classically named group codes, and more specifically $G$-codes; see \cite{MR0253797}.
From a coding theoretical point of view, it is meaningful to look for $G$-codes $\CC$ with a large $K$-dimension $\dim\CC$ and a large minimum (Hamming) distance $d(\CC)$.
Several remarkable codes can be detected as ideals in group algebras; for instance, this holds for the extended binary Golay code
 \cite{bernhardt1990extended} - which is related to the Leech lattice, to the sporadic simple group ${\rm M}_{24}$ and to various design-theoretic objects - and for binary Reed-Muller codes \cite{MR270816} - which have strong connections to geometry.
 Moreover, $G$-codes over $K$ have been proved to be asymptotically good for any finite field $K$, see \cite{bazzi2006some,MR4145573}: there exist infinitely many groups $G$ (of growing order) and ideals $\CC\leq KG$ with both large dimension and large minimum distance (where ``large'' means linear in the order of $G$).
 The algebraic structure of $G$-codes has been intensively studied; see e.g. \cite{bernal2009intrinsical,borello2021checkable,borello2019algebraic,garcia2020dimension} and the references therein. Yet, there are still many open questions about their coding theoretical properties. 
After recalling some notations and preliminary results in Section \ref{sec:prelim}, the aim of this paper is twofold: Section \ref{sec:UP} deals with a bound on the coding-theoretical parameters of a $G$-code, while Section \ref{sec:Schur} investigate the structure of $G$-codes in relation to their Schur product. We now give some more details on Sections \ref{sec:UP} and \ref{sec:Schur}.

In Section \ref{sec:UP}, we generalize in Theorem \ref{thm:UP} an uncertainty principle proved by Meshulam \cite{meshulam1992uncertainty} to $K$-valued functions over $G$ for any field $K$ and finite group $G$, and we put it in the context of coding theory.
Uncertainty principles for functions $f$ over abelian groups are classical harmonic analytic results assuring that either $f$ or its Fourier transform $\hat{f}$ has large support; see \cite{folland1997uncertainty,tao2005uncertainty}. Recentely, Evra, Kowalski and Lubotzky \cite{MR3852174} have started to build a bridge between some uncertainty principles and the goodness of cyclic codes, which are indeed ideals in the group algebra over a cyclic group. The paper \cite{MR4324212} pushed forward with this link in relation to MDS codes and the BCH bound, while Feng, Hollmann and Xiang \cite{feng2019shift} extended the investigation to abelian groups. As a consequence of an uncertainty principle, we prove in Corollary \ref{coro:bound}  the bound
\begin{equation}\label{eq:bound}d(\CC)\cdot \dim \CC\geq |G|\end{equation}
 for any $G$-code $\CC$. Up to our knowledge, this is the first bound of this shape on the parameters of a general $G$-code. Note that, for certain families of linear codes in $K^n$, interesting results on the product of the minimum distance and the dimension have been recently obtained in \cite{alizadeh2021sequences}.
The rest of Section \ref{sec:UP} thoroughly describes  the structure of $G$-codes attaining equality in \eqref{eq:bound}.

In Section \ref{sec:Schur}, we define the Schur product in a group algebra $KG$ componentwise, in analogy with the Schur - or Hadamard - product of matrices, which is object of the celebrated Schur product theorem on positive definite matrices \cite{Schur}.
We then define the Schur product of $G$-codes $\CC$ as the $K$-linear subspace spanned by the Schur product of their elements, and investigate it as a $G$-code itself.
In particular we focus on the dimension of the Schur square $\CC*\CC$, boosted by two main reasons.
At first, as highlighted in \cite{MR3364442}, the dimension of $\CC*\CC$ is related to the Hilbert sequence and the Castelnuovo-Mumford regularity of $\CC$, which are defined via the classical correspondence between $k$-dimensional linear codes over $K$ and point multisets in the $(k-1)$-dimensional projective space over $K$. Secondly, as noticed in \cite{couvreur2014distinguisher}, the dimension of the Schur square of structured codes $\CC$ can be quite small, in contrast to  random codes, and this has relevant consequences for the security of code-based cryptosystems related to $\CC$.
Section~\ref{sec:Schur} considers the module structure of the Schur square $\CC*\CC$ of $G$-codes $\CC$. Theorem \ref{thm:induced} is a structure result in the case $\CC*\CC=\CC$. Then we show in Theorem \ref{thm:Schur} that the projective cover of the trivial module is a direct summand of $\CC*\CC$ whenever $\CC$ is not self-orthogonal. Finally, we explore some consequences of Theorem \ref{thm:Schur}, with a particular attention on the case of $p$-groups in characteristic $p$, when non-self-orthogonal $G$-codes have Schur square of maximum dimension.

\section{Preliminary results}\label{sec:prelim}

Throughout the paper, the following notations are used:
\begin{itemize}
    \item $K$ is a field, and in particular $\F_q$ is the finite field of order $q$.
    \item $n$ is a positive integer, $K^n$ is the $n$-dimensional coordinate $K$-vector space, and we write $V\leq K^n$ for $K$-linear subspaces $V$ of $K^n$.
    \item $G$ is a finite group, $p$ is a prime number, $|G|_p\geq1$ is the largest power of $p$ dividing $|G|$, and we write $H\leq G$ for subgroups $H$ of $G$.
    \item $KG=\left\{\sum_{g\in G}a_g g\mid a_g\in K\right\}$ is the group algebra of $G$ over $K$, and we write $\CC\leq KG$ for \emph{right} ideals $\CC$ of $KG$.
    \item Functions from $G$ to $K$ are identified with elements of $KG$ via the correspondence between $f:G\to K$ and $\sum_{g\in G}f(g)g$.
    \item $K_G$ is the trivial $KG$-module.
    \item If $H\leq G$ and $M$ is a $KH$-module, then $M^G$ is the \emph{induced} $KG$-module, which is defined up to isomorphism as $M^G=\bigoplus_{i=1}^s Mg_i$, where $\{g_1,\ldots,g_s\}\subseteq G$ is a right transversal of $H$ in $G$.
    \item When considering trivial $KH$-modules with $H\leq G$, or $KG$-modules induced by $KH$-modules, or projective covers of $KH$-modules, we always identify them with the corresponding isomorphic right ideals of $KG$.
\end{itemize}

For other classical results on group algebras and its ideals, we refer the reader to \cite{HB} or
 \cite[Chapter 16]{huffman2021concise}.

\subsection{Linear codes in $K^n$}

We recall some basic notions on linear codes in $K^n$, i.e. on linear subspaces endowed with the Hamming metric; see \cite[Chapter 1]{huffman2021concise}.

The \emph{support} and the \emph{weight} of an element $v=(v_1,\ldots,v_n)\in K^n$ are defined  as
\[
\supp(v):=\left\{ i\in\{1,\ldots,n\}\mid v_i\ne0 \right\}\subseteq\{1,\ldots,n\}\  \text{and} \  \wt(v):=|\supp(v)|\in\{0,\ldots,n\}
\]  respectively,
and the \emph{Hamming distance} on $K^n$  as
\[
d(v,v'):=|\supp(v-v')|,\quad\mbox{for any }v,v'\in K^n.
\]
A \emph{linear code} $\CC$ of lenght $n$ over the alphabet $K$ is a $K$-linear subspace of $K^n$, endowed with the Hamming metric, and the elements of $\CC$ are called \emph{codewords}.
The \emph{minimum distance} $d(\CC)$ of $\CC$ is defined as the minimum Hamming distance between two distinct codewords, and coincides with the minimum weight of a nonzero codeword:
\[
d(\CC):=\min\{d(c,c')\mid c,c'\in\CC,c\ne c'\}=\min\{\wt(c)\mid c\in\CC,c\ne0\}.
\]
If $d=d(\CC)$ and $k=\dim(\CC)$, we denote the parameters of $\CC$ by $[n,k,d]$. 
A \emph{generator matrix} of $\CC$ is a $k\times n$ matrix $M$ over $K$ whose rows form a basis of $\CC$.
An \emph{information set} for $\CC$ is a $k$-subset $S$ of $\{1,\ldots,n\}$ such that the columns of $M$ indexed by $S$ are linearly independent (note that the property of being an information set does not depend on $M$).

An $[n,k,d]$-linear code $\CC'$ over $K$ is \emph{permutation equivalent} to $\CC$ if there exists an $n\times n$ permutation $n\times n$ matrix $P$ such that $M\cdot P$ is a generator matrix of $\CC'$; equivalently, if there exists a permutation $\sigma$ in the symmetric group $S_n$ such that \[\CC'=\sigma(\CC):=\{(c_{\sigma(1)},\ldots,c_{\sigma(n)})\in K^n\mid (c_1,\ldots,c_n)\in\CC\}.\]
Notice that permutation equivalence preserves the parameters of a code.
The \emph{permutation automorphism group} ${\rm PAut}(\CC)$ of $\CC$ is the stabilizer of $\CC$ in the action of $S_n$ by permutation equivalence, i.e. ${\rm PAut}(\CC)=\{\sigma\in S_n\mid \sigma(\CC)=\CC\}$.

With respect to the standard inner product $\langle(v_1,\ldots,v_n),(v_1',\ldots,v_n')\rangle:=\sum_{i=1}^n v_iv_i'$ on $K^n$ we consider the \emph{dual code} $\CC^{\perp}=\{v\in K^n\mid\langle v,c\rangle=0\mbox{ for all }c\in\CC\}$ of $\CC$.
The code $\CC$ is called \emph{self-orthogonal} if $\CC\subseteq\CC^{\perp}$ and \emph{self-dual} if $\CC=\CC^{\perp}$.

\begin{definition}
The \emph{Schur product} in $K^n$ is the bilinear map $K^n\times K^n\to K^n$ defined by
\[
(a_1,\ldots,a_n) * (b_1,\ldots,b_n) = (a_1b_1,\ldots, a_nb_n)\in K^n,\qquad\mbox{for any }a_i,b_i\in K,
\]
and the Schur product of two linear codes $\CC,\CC^{\prime}\leq K^n$ is defined as the linear code
\[
\CC * \CC'  = \langle c * c' \mid c \in \CC, c' \in \CC' \rangle_K  \leq K^n.
\]
\end{definition}

Since $\CC*\CC'=\langle c*c'\mid c\in B,c'\in B' \rangle_K$ whenever $\CC=\langle B\rangle_K$ and $\CC'=\langle B'\rangle_K$, we have
\[
 \dim \CC * \CC'   \leq  \min\left\{ n,\; \dim\CC\cdot\dim\CC' - \binom{\dim\CC \cap \CC^{\prime}}{2} \right\},
\]
and in particular
\begin{equation}\label{eq:BoundDim}
\dim \CC * \CC   \leq  \min\left\{ n, \binom{\dim\CC +1}{2} \right\}.
\end{equation}
By \cite[Proposition 5.2]{MCP}, if $\CC$ is chosen at random and $n>\binom{\dim\CC+1}{2}$, then $\dim \CC * \CC =\binom{\dim\CC+1}{2}$ almost surely. This is not the case for other algebraically structured codes, such as Reed-Solomon codes; see \cite[Section 5]{W}.

\subsection{$G$-codes in $KG$}

In analogy to $K^n$, we endow also the group algebra $KG$ with the Hamming metric as follows. The \emph{support} and the \emph{weight} of an element $f=\sum_{g\in G}a_g g\in KG$ ($a_g\in K$) are defined  as $\supp(f):=\{g\in G\mid a_g\ne0\}\subseteq G$ and $\wt(f):=|\supp(f)|\in\{0,\ldots,|G|\}$ respectively.
The \emph{Hamming distance} between two elements $f,f'\in KG$ is $d(f,f'):=|\supp(f-f')|$.

\begin{definition} \cite[Chapter 16]{huffman2021concise}
A \emph{$G$-code} $\CC$ over the alphabet $K$ is a \emph{right} ideal of $KG$, endowed with the Hamming metric. The \emph{length} of $\CC$ is $|G|$, the \emph{dimension} of $\CC$ is its dimension as a $K$-linear subspace of $KG$, and its \emph{minimum distance} is
\[
d(\CC):=\min\{d(c,c')\mid c,c'\in\CC,c\ne c'\}=\min\{\wt(c)\mid c\in\CC,c\ne0\}.
\]
\end{definition}

We remark here that the choice of \emph{right} ideals as $G$-codes is a convention that does not affect the validity of the correspondent results for left ideals.

By choosing an ordering on $G$, say $G=\{g_1,\ldots,g_n\}$ with $n=|G|$, we may define a \emph{standard} $K$-linear isomorphism $\varphi:KG\to K^n$ by $\sum_{i=1}^na_i g_i\mapsto (a_1,\ldots,a_n)$.
We stress that such a standard isomorphism is unique up to the choice of the ordering.
The following characterization holds, and allows us to identify those linear codes in $K^n$ that correspond to $G$-codes.
\begin{proposition}[\!\!\cite{bernal2009intrinsical}]
A $K$-linear subspace $\CC$ of $KG$ is a $G$-code if and only if the permutation automorphism group ${\rm PAut}(\varphi(\CC))\leq S_n$ of the linear code $\varphi(\CC)\leq K^n$ contains a subgroup isomorphic to $G$ acting regularly in its induced natural action on the $n$ positions.
\end{proposition}

Via a standard isomorphism $\varphi$, we may (uniquely) define the standard inner product and the Schur product on $KG$ through the corresponding products on $K^n$, as
\[\langle f,f'\rangle:=\langle \varphi(f),\varphi(f')\rangle \quad  \text{ and } \quad  f * f'=\varphi(f) * \varphi(f'),\quad\mbox{ for any }f,f'\in KG.\]
Therefore, we can also define the Schur product of $G$-codes $\CC,\CC'\leq KG$ as the following $K$-linear subspace of $KG$:
\begin{equation}\label{eq:schurgroupcodes}
    \CC*\CC:=\langle c*c'\mid c\in \CC,c'\in\CC' \rangle_K.
\end{equation}

\section{An uncertainty principle for $G$-codes}\label{sec:UP}

In this section, we give an uncertainty principle for field-valued functions over $G$ and the corresponding bound on the parameters of a $G$-code, also investigating equality in this bound. \\

Following \cite[Section V]{feng2019shift}, we first introduce the following definition for any finite group $G$.

\begin{definition}
Let $S$ be a nonempty subset of $G$. A sequence  $g_1,\ldots,g_t$ in $G$ has  \emph{right $S$-rank} $t$ if $S g_i=\{s g_i\mid s\in S\}$ is not contained in $\bigcup_{j<i}Sg_i$ for all $i\in\{2,\ldots,t\}$. 
\end{definition}

For any $f\in KG$, we denote by $T_f:KG\to KG$ the map $v\mapsto fv$.

\begin{lemma}\label{lemma:dim}
Let $0\neq f\in KG$ and $S=\supp(f)$. If there exists a sequence in $G$ with right $S$-rank $t$, then $\dim fKG=\rank_K(T_f)\geq t$.
\end{lemma}

\begin{proof}
Let $g_1,\ldots,g_t$ be a sequence in $G$ of $S$-rank $t$.
We aim to show that $T_f(g_1),\ldots,T_f(g_t)\in KG$ are linearly independent over $K$.
Suppose by contradiction that there exists $(\lambda_1,\ldots,\lambda_t)\in K^t\setminus\{0\}$ such that $\sum_{i=1}^{t}\lambda_i T_f(g_i)=0$.
Let $r:=\max\{i\in\{1,\ldots,t\}\colon \lambda_i\ne0\}$, so that $\sum_{i=1}^{t}\lambda_i T_f(g_i)$ is a linear combination of $\bigcup_{i=1}^{r} S g_i$.
Let $h\in S g_r \setminus \bigcup_{j<r} S g_j$.
Then the coefficient of $h$ in $\sum_{i=1}^{t}\lambda_i T_f(g_i)\in KG$ is $\lambda_r\ne 0$, a contradiction.
\end{proof}

\begin{remark}
We highlight the fact that the lower bound on $\dim fKG$ in Lemma \ref{lemma:dim} is an instrinsic property of $G$ and does not depend on $K$. More precisely, the bound depends on $\supp(f)\subseteq G$ but not on the nonzero coefficients in $f$ of the elements in $\supp f$.
\end{remark}

Following the arguments of \cite[Theorem 1.(a)]{meshulam1992uncertainty} we prove Theorem \ref{thm:UP}, that holds for any field $K$ and any finite group $G$.

\begin{theorem}\label{thm:UP}
For any $f\in KG$, define $T_f\colon KG\to KG$ by $v\mapsto fv$.
Then
$$
|\supp(f)|\cdot\rank_K(T_f)\geq |G|.
$$
\end{theorem}

\begin{proof}
Let $S:=\supp(f)\subseteq G$.
Let $t\geq 1$ be the maximum size of a sequence $g_1,\ldots,g_t\in G$ with right $S$-rank $t$.
By the maximality of $t$, we have $\bigcup_{i=1}^{t}S g_i=G$ and hence $t\geq |G|/|S|$.
We may then conclude by Lemma \ref{lemma:dim}.
\end{proof}

\begin{remark}
Suppose that ${\rm char}K$ does not divide the order of $G$.
Given a representation $\rho$ of $G$ over $K$, the Fourier transform of $f\in KG$ at $\rho$ is
$$\hat{f}(\rho):=\sum_{g\in G} f(g)\rho(g).$$
Let ${\rm Irr}(G)=\{\rho_1,\ldots,\rho_t\}$ be the set of irreducible representations of $G$ over $K$.
Let $\supp(\hat{f}):=\{\rho\in{\rm Irr}(G)\mid \hat{f}(\rho)\neq 0\}$.
Since $KG$ is semisimple, the map 
$$\varphi:
h\to (\hat{h}(\rho_1),\ldots,\hat{h}(\rho_t))$$
is an isomorphism \cite[Proposition 10]{serre1977linear}. Define $S:=\varphi\circ T_f\circ\varphi^{-1}$. Then, after the notation of \cite{meshulam1992uncertainty},
$$\mu(f):=\sum_{i=1}^t \deg \rho_i \cdot \rank_K \hat{f}(\rho_i)=\rank_K S=\rank_K T_f.
$$
So Theorem \ref{thm:UP} reads 
$$|\supp(f)|\cdot \mu(f)\geq |G|,$$
and hence $\mu(f)$ is a measure of $\supp(\hat{f})$.
For this reason, Theorem \ref{thm:UP} is a kind of uncertainty principle for the function $f:G\to K$.
\end{remark}

\begin{corollary}\label{coro:bound}
For any nonzero $G$-code $\mathcal{C}$, we have 
\[d(\mathcal{C})\cdot \dim\mathcal{C}\geq |G|.\]
In particular,
$$ 2\sqrt{|G|}\leq d(\CC)+\dim\CC\leq |G|+1. $$
\end{corollary}

\begin{proof}
    Let $f\in\mathcal{C}$ be such that $d(\mathcal{C})=|\supp(f)|$. Then the subcode $fKG$ has the same minimum distance as $\mathcal{C}$, but possibly a smaller dimension. Since $\dim(fKG)=\rank_K(T_f)$, the first claim follows by Theorem \ref{thm:UP}.
    The second claim is a consequence of the AM-GM inequality and the Singleton bound.
\end{proof}

\begin{remark}
When $G$ is cyclic, the structure of the defining zeros of the cyclic $G$-code $\CC$ is used to prove the BCH bound on the minimum distance. This idea can be extended to prove the so-called shift bound whenever $G$ is abelian of order coprime to ${\rm char}K$, by considering the defining zeros of $\CC$ as particular elements $\hat{f}$ of the character group $\hat{G}$; see \cite[Section 3]{feng2019shift}. Notice that, when $f\in\CC$, $\hat{f}$ is related to the dimension of the submodule $fKG$ of $\CC$.
We are not able to extend further this strategy to non-abelian groups. Yet one may ask how to define ``zeros'' of $\CC$ in relation to the submodules of $\CC$, and hence to $\dim\CC$. In this way, \eqref{eq:bound} may be read as an analogous bound on the minimum distance.
\end{remark}

\begin{remark}
Note that Theorem \ref{thm:UP} and Corollary \ref{coro:bound} are a generalization of the Naive uncertainty principle proved in \cite[Proposition 2]{MR4324212} for cyclic groups $G$.
It is then natural to wonder whether an analogue of \cite[Theorem 2]{MR4324212} on the asymptotic behavior of cyclic codes holds also for other families of $G$-codes.
For instance, an even stronger result holds for $G$-codes when $G$ is metacyclic, since metacyclic codes are asymptotically good; see \cite{bazzi2006some,MR4145573}. However, it is likely that results similar to \cite[Theorem 2]{MR4324212} may hold for different families of groups. 
\end{remark}

\begin{example}\label{ex:golayandRM}
If $\CC$ is the extendend binary Golay code, then $\CC$ is a self-dual code of length $24$ over $\mathbb{F}_2$ and also an $S_4$-code in $\mathbb{F}_2 S_4$; see \cite{bernhardt1990extended}. In this case $d(\CC)\cdot \dim \CC=8\cdot 12=96> |G|$.

If $\CC={\rm RM}(r, m):=\{(f(v))_{v\in \F_2^m}\mid f\in \F_2[x_1,\ldots,x_m], \ \deg f\leq r\}$ is the binary Reed-Muller code of order $r$ in $m$ variables (with $r\leq m$), then $\CC$ is a $G$-code, where $G$ is an elementary abelian $2$-group of rank $m$; see \cite{MR270816}.
In this case, 
\[d(\CC)\cdot \dim \CC=2^{m-r}\cdot \sum_{i=0}^r\binom{m}{i}\geq 2^{m-r}\cdot \sum_{i=0}^r\binom{r}{i}=2^m=|G|,\]
and the equality holds if and only if $r=m$.
\end{example}

We characterize the case in which equality in \eqref{eq:bound} holds, generalizing \cite[Theorem 1.(b)]{meshulam1992uncertainty}.

\begin{theorem}\label{th:lower} 
A $G$-code $\CC$ satisfies $d(\CC)\cdot \dim\CC=|G|$ if and only if there exist $H \leq G$ and $c\in KH$ such that $|H|= d(\CC)$, $cKH$ has dimension $1$ and $\CC=cKG$.
\end{theorem}

\begin{proof}
If there exist $H$ and $c$ as in the claim, then $\mathcal{C}$ is induced by the $KH$-module $cKH$, with exactly $[G:H]$ direct summands, each of dimension $1$. Thus $\dim\CC=[G:H]$ and $d(\CC)\cdot\dim\CC=|G|$.

Conversely, suppose that $d(\CC)\cdot \dim\CC=|G|$.
Let $c\in \CC$ be such that $\wt(c)=d(\CC)$ and consider $\CC_0:=cKG\leq \CC$. Clearly $d(\CC_0)=d(\CC)$. Thus, by Corollary \ref{coro:bound},
$$|G|\leq d(\CC_0)\cdot \dim\CC_0=d(\CC)\cdot \dim\CC_0\leq d(\CC)\cdot \dim\CC=|G|.$$
It follows that $\dim\CC_0=\dim\CC$, and hence $\CC=\CC_0=cKG$. Let $H:=\supp(c)$, so that $|H|=d(\CC)$.
Following the arguments and the notations of the proof of Theorem \ref{thm:UP} with $f=c$,
we see that $t=|G|/|H|$ and hence $G$ is the disjoint union
$$ G=Hg_1\sqcup\cdots\sqcup Hg_t, $$
for some $g_1,\ldots,g_t\in G$.
Replacing $c$
by the minimum weight codeword $ch^{-1}\in\CC$ with $h\in\supp(c)$, we can assume $1\in H$.
Then, for any $h\in H$, we have $h\in H\cap Hh$, and this implies $H=Hh$ by the following argument from the proof of Theorem 1 in \cite{meshulam1992uncertainty}. Suppose by contradiction that $Hh\not\subseteq H$, and choose a maximal sequence $h_1,\ldots,h_r\in G$ with $r\geq2$ such that $h_1=1$, $h_2=h$ and $H h_i\not\subseteq \cup_{1\leq j<i}H h_j$ for $2\leq i\leq r$. By maximality $G=\cup_{1\leq j\leq r}Hh_j=G$, and hence $t\geq r$ by definition of $t$ as in the proof of Theorem \ref{thm:UP}. But $H\cap Hh_1$ implies $r>|G|/|H|=t$, a contradiction.
Thus $Hh=H$ for all $h\in H$, whence $H\cdot H\subseteq H$. Since $H$ is finite, this implies that $H$ is a subgroup of $G$, and $t=[G:H]$.
Since
$$
\CC=cKG=cK\left(\bigsqcup_{i=1}^tHg_i\right)=\bigoplus_{i=1}^t (cKH)g_i,
$$
we have $\dim\CC=t\cdot\dim cKH$. By the assumption, $\dim\CC=|G|/d(\CC)=t$. Therefore $\dim cKH=1$, and the claim is proved.
\end{proof}

If $H$  in the claim of Theorem \ref{th:lower} is a $p$-group and ${\rm char}K=p$, then the simple $KH$-module $cKH$ is trivial. Therefore Theorem \ref{th:lower} yields immediately the following result.

\begin{corollary}\label{cor:boundp}
If ${\rm char}K=p$ and $G$ is a $p$-group, then a $G$-code $\CC\leq KG$ satisfies $d(\CC)\cdot\dim\CC=|G|$ if and only if $\CC=K_H^G$ for some $H\leq G$.
\end{corollary}

\begin{remark}
The claim of Corollary \ref{cor:boundp} does not hold if $G$ is not a $p$-group where $p={\rm char}K$. For instance, if $K=\F_3$, $G=C_2=\langle r\rangle$, $c=1+2r$ and $\CC=cKG$. Then 
Theorem \ref{th:lower} holds with $H=G$, but $\CC\ne K_H^G$.
\end{remark}

Among the $G$-codes attaining equality in \eqref{eq:bound}, Proposition \ref{prop:loweridemp} identifies those which are generated by an idempotent.

\begin{proposition}\label{prop:loweridemp}
Let ${\rm char}K=p$ and $\CC\leq KG$ be a $G$-code such that $d(\CC)\cdot \dim\CC=|G|$. Then $\CC=eKG$, for some $e\in KG$ with $e=e^2$, if and only if $p\nmid d(\CC)$.
\end{proposition}
\begin{proof}
Suppose that $\CC=eKG$ with $e=e^2$.
Since $\CC$ is a projective $KG$-module, we have $|G|_p \mid \dim \CC$ by Dickson's Theorem \cite[Chapter 7, Corollary 7.16]{HB}. Thus $d(\CC)=|G|/\dim\CC$ is not divisible by $p$.

To see the converse we apply Theorem \ref{th:lower}. Let $H$ and $c$ be as in the claim of Theorem \ref{th:lower}. Since $p\nmid d(\CC)=|H|$, the algebra $KH$ is semisimple, and hence the ideal $cKH$ of $KH$ is generated by an idempotent $e\in KH$, i.e. $cKH=eKH$. It follows that $\CC = cKG =e KG$.
\end{proof}

Let us finally remark that, if $\CC<\F_2G$ and $G$ is a $2$-group, then $d(\CC)$ is even, since $\CC$ is contained in the Jacobson radical ${\rm J}(\F_2G)$, which is the subspace of even weight vectors.

\section{On the Schur product of  $G$-codes}\label{sec:Schur}

In this section we deal with the Schur product of $G$-codes, which is  defined in  \eqref{eq:schurgroupcodes}.

\begin{lemma}\label{lemma:first}
If $\CC,\CC'\leq KG$ are $G$-codes, then $\CC * \CC' \leq KG$ is a $G$-code as well.
\end{lemma}
\begin{proof}
 Let  $c = \sum_{x \in G}c_x x \in \CC \leq KG $ and $c' = \sum_{x \in G}c'_x x \in \CC' \leq KG$, with $c_x,c_x^{\prime}$ for all $x\in G$. 
For any $g \in G$, we have
$$ \sum_{x \in G}c_{xg^{-1}}x = \sum_{x \in G}c_x xg = cg \in \CC $$ and 
$$ \sum_{x \in G}c'_{xg^{-1}}x = \sum_{x \in G}c'_x xg = c'g \in \CC' $$
since $\CC$ and $\CC'$ are  $KG$-modules. Hence
$$ (c * c')g = \left(\sum_{x \in G} c_x c'_x x\right)g = \sum_{x \in G} c_x c'_x xg = \sum_{x\in G} c_{xg^{-1}}c'_{xg^{-1}} x =cg* c'g \in \CC * \CC'$$
which proves the claim.
\end{proof}

By Lemma \ref{lemma:first}, we can investigate $KG$-linar maps involving the $KG$-module $\CC*\CC'$.

\begin{lemma} Let $\CC, \CC' \leq KG$ be $G$-codes of dimension $k$ and $k'$, respectively. For any $u\in\CC'$, let $\varphi_u:\CC\to \CC * \CC'$ be the map $c\mapsto c* u$. Then the following holds.
\begin{itemize}
    \item[a)] For any $u\in \CC'$, $\varphi_u$ is $K$-linear.
    \item[b)] If $k\leq k'$, then there exists $v \in \CC'$ such that $\varphi_{v}$ is injective.
    \item[c)] For any $w\in\CC'$, the map $\varphi_w$ is $KG$-linear if and only if $w$ belongs to the trivial $KG$-module $K_G$. Also, for any nonzero $w\in\CC'\cap K_G$, $\varphi_w$ is injective.
\end{itemize}
\end{lemma}

\begin{proof} a) This claim is straightforward.\\
b) 
 Let $S$ be an information set for $\CC$. Since the existence of $v$ as in the claim is invariant under permutation equivalence for $\CC'$, we can assume that $\CC'$ has an information set $S'$ containing $S$. Let $M'$ be a generator matrix for $\CC'$ such that the columns of $M'$ indexed by $S'$ form the identity matrix $I_{k^\prime}$, and choose $v\in\CC'$ as the sum of the rows of $G'$. Then, for any $c\in\CC$ and any position $i$ in $S$, the $i$-th entries of $c$ and $\varphi_v(c)$ are equal. The claim follows. \\
c) Let $w=\sum_{x\in G}\lambda_x x\in\CC'$. If $w\in K_G$, say $\lambda_x=\lambda\in K$ for any $x\in G$, then
$$\varphi_w(cg) = cg*w=cg*wg= (c*w)g = \varphi_w(c)g$$
for all $c\in\CC$ and $g\in G$. Since $\varphi_w$ is $K$-linear, this implies that $\varphi_w$ is $KG$-linear.
Conversely, suppose that $\varphi_w$ is $KG$-linear. Then
$$cg*w =\varphi_w(cg) =  \varphi_w(c)g =
(c*w)g = cg * wg$$
for all $c \in \CC$ and $g \in G$.
Thus $c*w = c *wg$ and hence
$ c*(w-wg)=0$ for all $c \in \CC$ and $g \in G$.
If there exist $x,g\in G$ such that the component of $w -wg$ at $x$ is nonzero, then we choose $c \in \CC$ with nonzero component at $x$ and obtain a contradiction to $c*(w-wg)=0$. Therefore $w=wg$ for all $g\in G$, implying that $\lambda_x=\lambda\in K$ for all $x\in G$, i.e. $w\in K_G$.
Finally, if $w\in\CC'\cap K_G$ and $w\ne0$, then $\wt(w)=|G|$, and therefore $\varphi_w$ is injective.
\end{proof}

In general, not all $KG$-monomorphisms from $\CC$ into $\CC*\CC$ have the shape $\varphi_{u}$. 

\begin{remark} \label{binary}
If $K=\F_2$, then the map $c\mapsto c* c$ is a $KG$-monomorphism from $\CC$ into $\CC*\CC$.
\end{remark}

On the other side, $KG$-monomorphisms from $\CC$ into $\CC*\CC$ may not exist at all, as  Example~\ref{ex:nomono} shows.

\begin{example}\label{ex:nomono}
Let $K=\F_3$, $G=C_2=\langle r\rangle$ and $\CC=(1+2r)KG$. Then $\CC$ is a nontrivial irreducible $KG$-module of dimension $1$. By direct computation we have
$\CC * \CC=(1+r)KG$, which is the trivial $KG$-module. Therefore a $KG$-monomorphism from $\CC$ into $\CC* \CC$ does not exist.
\end{example}

Theorem \ref{thm:induced} provides the structure of those $G$-codes $\mathcal{C}$ which coincide with $\CC*\CC$.

\begin{theorem} \label{thm:induced}
Let $\CC\leq KG$ be a $G$-code with $\CC\ne\{0\}$.
 If $\CC=\CC * \CC$, then there exists a subgroup $H\leq G$ such that $\CC=K_H^G$. In particular, $d(\CC)\cdot \dim(\CC)=|G|$.
\end{theorem}

\begin{proof}
Any two codewords $c,c'\in\CC$ of minimum weight $\wt(c)=\wt(c')=d(\CC)$ satisfy 
\begin{equation}\label{eq:disjoint}
\supp(c)\cap \supp(c')=\emptyset\qquad\mbox{or}\qquad \supp(c)\cap\supp(c')= \supp(c),
\end{equation}
because $c*c'\in\CC$ and $\supp(c*c')\subseteq\supp(c)$. 

Moreover, every element of $G$ is in the support of some minimum weight codeword of $\CC$, because the multiplication action of $G$ on itself is transitive and $\CC$ is a $KG$-module.
Therefore, there exist $c_1,\ldots,c_s\in\CC$ of minimum weight such that $G$ is the disjoint union 
$$G=\supp(c_1)\sqcup  \cdots \sqcup \supp(c_s).$$

Let $O_i=\supp(c_i)$ for $i=1,\ldots,s$, and assume without loss of generality that $1\in O_1$.
Then, for any $g\in O_1$, the set $O_1g$ is the support of the minimum weight codeword $c_1g$ and satisfies $g\in O_1\cap O_1g$, which implies $O_1=O_1g$ by \eqref{eq:disjoint}. 
Thus $O_1=O_1O_1$, and hence $H:=O_1$ is a subgroup of $G$. 

For any $i=1,\ldots,s$, it follows that $O_i=Hg_i$ for some $g_i\in G$ (indeed, for $g_i\in O_i)$. 

Write $c_i=\sum_{h\in H}\lambda_h^{(i)} hg_i$, where $\lambda_h^{(i)}\not=0$ for all $h\in H$.
For any $h^\prime \in H$, the codeword $c_i* c_i-\lambda_{h^\prime}^{(i)} c_i\in\CC$ has weight $|\{h\in H\colon \lambda_h^{(i)}\ne\lambda_{h^\prime}^{(i)}\}|<\wt(c_i)=d(\CC)$, and hence $c_i* c_i-\lambda_{h^\prime}^{(i)} c_i=0$.
Then $\lambda_h^{(i)}=\lambda^{(i)}$ is constant for all $h\in H$ and, after scaling by $(\lambda^{(i)})^{-1}\in K$, we can assume that $c_i=\sum_{h\in H}h g_i$.
It follows that 
$$\left(K_H\right)^G=\bigoplus_{i=1}^s\left(K\sum_{h\in H}h\right)g_i=\bigoplus_{i=1}^s K\left(\sum_{h\in H}h g_i\right)=\bigoplus_{i=1}^s K c_i\subseteq \CC.$$
Conversely, in order to prove that $\CC\subseteq K_H^G$, let $a\in\CC$. Then $$a=a*\sum_{g\in G}g=a*\sum_{i=1}^s c_i=\sum_{i=1}^s a* c_i=\sum_{i=1}^s a_i,$$
where $a_i:=a*c_i\in KHg_i$.
From $\CC * \CC\subseteq\CC$ it follows $a_i\in\CC$. Since $\supp(a_i)\subseteq\supp(c_i)$, this implies that either $a_i=0$ or $\wt(a_i)=d(\CC)$. Then, arguing as above, we have $a_i=\mu^{(i)} c_i$ for some $\mu^{(i)} \in K$.
Hence $a=\sum_{i=1}^s \mu^{(i)}c_i\in K_H^G$ and the equality $\CC=K_H^G$ is proved.

Finally, the claim $d(\CC)\cdot \dim(\CC)=|G|$ follows immediately from Theorem \ref{th:lower}.
\end{proof}

For any $G$-code $\CC \leq KG$, we define recursively $\CC^{(1)}=\CC$ and $\CC^{(t+1)}=\CC^{(t)}*\CC$.
By \cite[Theorem 2.32]{MR3364442}, we know that
$\dim \CC^{(t+1)} \geq \dim \CC^{(t)}$ for any $t \geq 1$.
This allows to define the \emph{Castelnuovo-Mumford regularity} of $\CC$ as the smallest $t$ such that $\dim\CC^{(t+i)}=\dim\CC^{(t)}$ for all $i\geq0$; see \cite[Definition 1.5]{MR3364442}. 
The eventual behaviour of $\CC^{(t)}$ is easily obtained in the binary case.

\begin{theorem}\label{thm:regularity}
Let $K=\F_2$ and $\CC \leq KG$ be a $G$-code with $\CC\ne\{0\}$.
Then there exists a subgroup $H\leq G$ such that $\CC\leq K_H^G$ and $\CC^{(t)}=K_H^G$ for any $t$ big enough.
\end{theorem}
\begin{proof}
By Remark \ref{binary}, we know that $ \CC^{(2^i)} \leq \CC^{(2^{i+1})}$ for any $\geq0$.
Since $G$ is finite, there exists $t\in\mathbb{N}$ such that $\CC^{(2^{t})}*\CC^{(2^{t})}=\CC^{(2^{t+1})}=\CC^{(2^{t})}$.
By Theorem \ref{thm:induced}, this implies that $\CC^{(2^t)}=K_H^G$ for some $H\leq G$.
Note that the case $\CC^{(2^t)}=K G$ corresponds to $H=1$.
\end{proof}

In the next results we investigate the $KG$-module $\CC*\CC$ in relation with the self-orthogonality of $\CC$.
Notice that the standard inner product on $K^n\times K^n$ corresponds via a standard isomorphism $K^n\cong KG$ to the inner product on $KG\times KG$ defined by
\[
\langle\cdot,\cdot\rangle : KG\times KG \to K,\quad (c,c')\mapsto \varepsilon(c*c'),
\]
where $\varepsilon$ is the $KG$-linear augmentation defined by
\[
\varepsilon : KG\to K,\quad \sum_{g\in G}a_g g\mapsto \sum_{g\in G}a_g.
\]
Therefore, a $G$-code $\CC\leq KG$ is self-orthogonal if and only if $\CC*\CC\subseteq \ker\varepsilon$.

\begin{theorem}\label{thm:Schur}
Let $\CC\leq KG$ be a $G$-code such that $\CC\nsubseteq \CC^\perp$. Then
$\CC * \CC = \mathcal{P}_0 \oplus \mathcal{M}$, where $\mathcal{P}_0$ and $\mathcal{M}$ are $KG$-modules, and
$\mathcal{P}_0$ is 
a projective cover of the trivial $KG$-module $K_G$. 
In particular, if ${\rm char}K=p$, then $|G|_p\leq \dim \CC * \CC$.
\end{theorem}

\begin{proof}
Write $KG=\mathcal{P}_0\oplus \mathcal{P}_1$,
where $\mathcal{P}_0$ is
a projective cover of $K_G$. 
For all $a\in KG$ denote by $a_0$ and $a_1$ the components of $c$ in $\mathcal{P}_0$ and $\mathcal{P}_1$, respectively, and consider the projection $\pi:\CC*\CC\to\mathcal{P}_0$ defined by $a\mapsto a_0$. We prove that $\pi$ is $KG$-linear and surjective.

Note that $\ker \varepsilon =\mathcal{P}_0{\rm J}(KG)\oplus \mathcal{P}_1$, where ${\rm J}(KG)$ is the Jacobson radical of $KG$; see \cite[Chapter 7,\ \S  10]{HB}.
Also, since $\CC\nsubseteq \CC^\perp$, there exists $c\in(\CC*\CC)\setminus\ker\varepsilon$, which satisfies $c_0=\pi(c)\in\mathcal{P}_0\setminus\ker\varepsilon$.
Therefore $c_0\in\mathcal{P}_0\setminus\mathcal{P}_0{\rm J}(KG)$, and we get $\mathcal{P}_0=c_0 KG$.
Thus $\pi$ is a $KG$-epimorphism from $\CC * \CC$ onto the projective $KG$-module $\mathcal{P}_0$.  Hence, up to a $KG$-isomorphism, $\mathcal{P}_0$ is a direct summand of $\CC * \CC$; see \cite[Chapter 7, \S 7]{HB}.

The claim on $|G|_p$ now follows by $\dim\mathcal{P}_0\leq\dim\CC*\CC$ and Dickson's Theorem; see \cite[Chapter 7, Corollary 7.16]{HB}.
\end{proof}

Proposition \ref{prop:solv} gives an application of Theorem \ref{thm:Schur} to $G$-codes over the binary field.

\begin{proposition}\label{prop:solv}
Let $K=\F_2$ and $\CC\leq\F_2 G$ be a $G$-code such that $\CC\not\subseteq\CC^{\perp}$.
Suppose that the projective cover $\mathcal{P}_0$ of the trivial $KG$-module is 
$K_H^G$ for some subgroup $H\leq G$; for instance, this holds if $H$ is a normal $p$-complement of $G$.
Then $\CC*\CC=\mathcal{P}_0$ if and only if $\CC\leq\mathcal{P}_0$.
\end{proposition}
\begin{proof}
 By Remark \ref{binary}, $\CC$ is a submodule of $\CC*\CC$.
 Therefore, $\CC*\CC=\mathcal{P}_0$ implies $\CC\leq\mathcal{P}_0$.
 If $G$ is $2$-nilpotent and $H$ is a normal $2$-complement of $G$, then $\mathcal{P}_0=K_H^G$ because ${\rm char}K=2$; see \cite[Chapter 7, \S 7]{HB}.
 Therefore we can suppose that $\mathcal{P}_0=K_H^G$ for some subgroup $H\leq G$. 
Since $K=\F_2$, it is easy to see that $\mathcal{P}_0*\mathcal{P}_0=\mathcal{P}_0$, whenever $\mathcal{P}_0=K_H^G$ with $H\leq G$.
Therefore, from $\CC\leq\mathcal{P}_0$ it follows that $\CC*\CC\leq\mathcal{P}_0*\mathcal{P}_0=\mathcal{P}_0$.
By Theorem \ref{thm:Schur}, $\CC * \CC$ contains a module
isomorphic to $\mathcal{P}_0$. Thus, from $\CC\leq\mathcal{P}_0$ it follows $\CC * \CC = \mathcal{P}_0$, and the claim is proved.
\end{proof}

 \begin{example}\label{ex:mathieu}
 Let $K=\F_2$ and $G$ be the Mathieu group ${\rm M}_{11}$, of order $2^4 \cdot 3^2 \cdot 5 \cdot 11$. Let $\CC$ be the projective cover $\mathcal{P}_0$ of the trivial $KG$-module, which satisfies $\dim \CC =2^4 \cdot 7$.
 Since $\dim\CC\nmid |{\rm M}_{11}|$, $\CC$ is not induced by the trivial module of a subgroup.
 By Theorem \ref{thm:induced}, we get $\CC \ne \CC * \CC$.
 Also, by direct checking, $\CC\not\subseteq\CC^{\perp}$.
Therefore, the claim of Proposition \ref{prop:solv} does not hold for $G={\rm M}_{11}$.
 \end{example}

\begin{lemma}\label{lemma:self}
If $\CC\leq KG$ is a self-orthogonal $G$-code, then $\dim\CC * \CC <|G|$.
\end{lemma}

\begin{proof}
From $\CC\subseteq \CC^\perp$ it follows that $\CC*\CC\leq\ker\varepsilon$.
With the notations of the proof of Theorem \ref{thm:Schur}, we have that $\ker\varepsilon =\mathcal{P}_0{\rm J}(KG)\oplus \mathcal{P}_1$ is strictly contained in $KG$, and the same holds for $\CC*\CC$. The claim follows.
\end{proof}

\begin{example}
Consider again the binary $G$-codes introduced in Example \ref{ex:golayandRM}.

If $\CC$ is the extended binary Golay code, then $\CC * \CC=\ker \varepsilon$, so that $\dim \CC * \CC=|G|-1=23$.

If $\CC$ is the Reed-Muller code ${\rm RM}(r, m)$ of order $r$ in $m$ variables with $r\leq (m-1)/2$, then $\CC$ is self-orthogonal. It is easy to observe that 
$\CC * \CC={\rm RM}(2r, m)$. In this case $\CC * \CC<\ker\varepsilon$ whenever $r<(m-1)/2$.
\end{example}

In the case of $p$-groups in characteristic $p$, the converse of Lemma \ref{lemma:self} holds. 

\begin{proposition}\label{prop:selfP}
Let ${\rm char}K=p$, let $G$ be a $p$-group and let $\CC\leq KG$ be a $G$-code.
If $\CC$ is not self-orthogonal, then $\CC*\CC=KG$.
\end{proposition}
\begin{proof}
The claim follows immediately from the last part of Theorem \ref{thm:Schur}.
\end{proof}

Using the bound \eqref{eq:BoundDim} on the dimension of $\CC*\CC$, Proposition \ref{prop:selfP} yields the following corollary.

\begin{corollary}
Let ${\rm char}K=p$, $G$ be a $p$-group and $\CC\leq KG$ be a $G$-code.
If
\[\dim\CC<\frac{\sqrt{8|G|+1}-1}{2},\]
then $\CC$ is self-orthogonal.
\end{corollary}

\section*{Acknowledgments}
The first author was partially supported by the ANR-21-CE39-0009 - BARRACUDA (French \emph{Agence Nationale de la Recherche}).
The third author was partially supported by the Italian National Group for Algebraic and Geometric Structures and their Applications (GNSAGA - INdAM).

\bibliographystyle{abbrv}
\bibliography{references.bib}

\end{document}